\documentclass{amsart}
\usepackage{amssymb}
\usepackage{amsmath}
\usepackage{amsfonts, latexsym}
\usepackage{color,graphics}
\usepackage{graphicx}

\newtheorem{theorem}{Theorem}[section]
\newtheorem{proposition}[theorem]{Proposition}
\newtheorem{lemma}[theorem]{Lemma}
\newtheorem{corollary}[theorem]{Corollary}
\newtheorem{definition}[theorem]{Definition}
\newtheorem{example}[theorem]{Example}

\begin{document}
\author{Hee Sun Jung$^1$ and Ryozi Sakai$^2$}
\address{$^{1}$Department of Mathematics Education, Sungkyunkwan University,
Seoul 110-745, Republic of Korea.}
\email{hsun90@skku.edu}
\address{$^{2}$Department of Mathematics, Meijo University, Nagoya 468-8502, Japan.}
\email{ryozi@crest.ocn.ne.jp}
\title[]{Higher order derivatives of approximation polynomials on $\mathbb{R}$}
\date{\today}
 \maketitle

\begin{abstract}
D. Leviatan has investigated the behavior of the higher order derivatives of approximation polynomials
of the differentiable function $f$ on $[-1,1]$.
Especially,  when $P_n$ is the best approximation of $f$,
he estimates the differences $\|f^{(k)}-P_n^{(k)}\|_{L_\infty([-1,1])}$,   $k=0,1,2,...$.
In this paper, we give the analogies for them with respect to the differentiable functions on $\mathbb{R}$,
and we apply the result to the monotone approximation.
\end{abstract}

Keywords; polynomial of the best approximation, exponential-type weight, monotone approximation

\setcounter{equation}{0}
\section{Introduction}

Let $\mathbb{R}=(-\infty,\infty)$ and ${\mathbb{R}}^+=[0,\infty)$.
We say that $f:  {\mathbb{R}}\rightarrow {\mathbb{R}^+}$ is quasi-increasing
if there exists $C>0$ such that $f(x)\leqslant Cf(y)$ for $0<x<y$.
The notation $f(x)\sim g(x)$ means that there are positive constants $C_1,  C_2$
such that for the relevant range of $x$, $C_1\leqslant f(x)/g(x)\leqslant C_2$.
The similar notation is used for sequences and sequences of functions.
Throughout $C,C_1,C_2,...$ denote positive constants independent of $n,x,t$.
The same symbol does not necessarily denote the same constant in different occurrences.
We denote the class of polynomials with degree $n$ by $\mathcal{P}_n$.

First, we introduce some classes of weights.
Levin and Lubinsky \cite{[8]} introduced the class of weights on ${\mathbb{R}}$ as follows.

\begin{definition}\label{Definition 1.1}
Let $Q:   \mathbb{R}\rightarrow [0,\infty)$ be a continuous even function, and satisfy the following properties:
\item[\,\,\,(a)]   $Q'(x)$ is continuous in $\mathbb{R}$, with $Q(0)=0$.
\item[\,\,\,(b)]   $Q''(x)$ exists and is positive in $\mathbb{R}\backslash\{0\}$.
\item[\,\,\,(c)]   $\lim_{x\rightarrow \infty}Q(x)=\infty.$
\item[\,\,\,(d)]   The function
\begin{equation*}
T_w(x):=\frac{xQ'(x)}{Q(x)},      x\neq 0
\end{equation*}
is quasi-increasing in $(0,\infty)$, with
\begin{equation*}
T_w(x)\ge \Lambda>1,      x\in \mathbb{R}\backslash\{0\}.
\end{equation*}
\item[\,\,\,(e)]   There exists $C_1>0$ such that
\begin{equation*}
  \frac{Q''(x)}{|Q'(x)|}\le C_1\frac{|Q'(x)|}{Q(x)},      \,\,\, a.e. \,\,\,   x\in \mathbb{R}.
\end{equation*}
Furthermore, if there also exists a compact subinterval $J (\ni 0)$ of $\mathbb{R}$, and $C_2>0$ such that
\begin{equation*}
  \frac{Q''(x)}{|Q'(x)|}\ge C_2\frac{|Q'(x)|}{Q(x)},      \,\,\, a.e. \,\,\,   x\in \mathbb{R}\backslash J,
\end{equation*}
then we write $w=\exp(-Q)\in \mathcal{F}(C^2+)$.
\end{definition}
For convenience, we denote $T$ instead of $T_w$, if there is no confusion.
Next, we give some typical examples of $\mathcal{F}(C^2+)$.
\begin{example}[{\cite{[5]}}]\label{Example 1.2} {\rm
(1) If $T(x)$ is bounded,
then we call the weight $w=\exp(-Q(x))$ the Freud-type weight
and we write $w\in\mathcal{F}^*\subset \mathcal{F}(C^2+)$.
\item[\,\,\,(2)] When $T(x)$ is unbounded,
then we call the weight $w=\exp(-Q(x))$ the Erd\"{o}s-type weight: (a) For $\alpha>1$,   $l\ge 1$ we define
\begin{equation*}
Q(x):=Q_{l,\alpha}(x)=\exp_l(|x|^{\alpha})-\exp_l(0),
\end{equation*}
where $\exp_l(x)=\exp(\exp(\exp\ldots\exp x)\ldots)      (l\,\, \textrm{times})$.
More generally, we define
\begin{equation*}
Q_{l,\alpha,m}(x)=|x|^m\{\exp_l (|x|^{\alpha}) -\tilde{\alpha}\exp_l(0)\},
\quad \alpha+m>1,\,\,   m\ge 0,\,\,   \alpha\ge 0,
\end{equation*}
where $\tilde{\alpha}=0$ if $\alpha=0$, and otherwise $\tilde{\alpha}=1$.
We note that $Q_{l,0,m}$ gives a Freud-type weight,
and $Q_{l,\alpha,m}$, ($\alpha >0$) gives an Erd\"{o}s-type weight.
\item[\,\,\,(3)] For $\alpha>1$,
$Q_{\alpha}(x)=(1+|x|)^{|x|^{\alpha}} -1 $ gives also an Erd\"{o}s-type weight.
}
\end{example}
For a continuous function $f : [-1,1] \to \mathbb{R}$, let
\begin{equation*}
E_n(f)=\inf_{P\in\mathcal{P}_n}\|f-P\|_{L_\infty([-1,1])}=\inf_{P\in\mathcal{P}_n}\max_{x\in[-1,1]}|f(x)-P(x)|.
\end{equation*}
D. Leviatan \cite{[7]} has investigated the behavior of
the higher order derivatives of approximation polynomials for the differentiable function $f$ on $[-1,1]$,
as follows:

\noindent
{\bf Theorem} (Leviatan \cite{[7]}). {\it
For $r\ge 0$ we let  $f\in C^{(r)}[-1,1]$,
and let $P_n\in \mathcal{P}_n$ denote the polynomial of best approximation of $f$ on $[-1,1]$.
Then for each $0\le k\le r$ and every $-1\le x\le 1$,
\begin{equation*}
\left|f^{(k)}(x)-P_n^{(k)}(x)\right|\le \frac{C_r}{n^k}\Delta_n^{-k}(x)E_{n-k}\left(f^{(k)}\right),
\quad     n\ge k,
\end{equation*}
where $\Delta_n(x):=\sqrt{1-x^2}/n+1/n^2$ and $C_r$ is an absolute constant which depends only on $r$.
}

 In this paper,  we estimate $\left|\left(f^{(k)}(x)-P_{n;f}^{(k)}(x)\right)w(x)\right|$,
 $x\in\mathbb{R}$,  $ k=0,1,...,r$
for $f\in C^{(r)}(\mathbb{R})$ and for some exponential type weight $w$ in $L_p(\mathbb{R})$-space,
$1 < p\le \infty$,
where $P_{n;f}\in \mathcal{P}_n$ is the best approximation of $f$.
Furthermore, we give an application for a monotone approximation with linear differential operators.
In Section 2 we write the theorems in the space $L_\infty(\mathbb{R})$,
then we also denote a certain assumption and some notations which need to state the theorems.
In Section 3 we give some lemmas and the proofs of theorems.
In Section 4, we consider the similar problem in $L_p(\mathbb{R})$-space, $1<p<\infty$.
In Section 5, we give a simple application of the result to the monotone approximation.

\setcounter{equation}{0}
\section{Theorems and Preliminaries}
First, we introduce some well-known notations.
If $f$ is a continuous function on $\mathbb{R}$, then we define
\begin{equation*}
  \left\|fw\right\|_{L_\infty(\mathbb{R})}:=\sup_{t\in\mathbb{R}}|f(t)w(t)|,
\end{equation*}
and for $1\le p<\infty$ we denote
\begin{equation*}
  \|fw\|_{L_p(\mathbb{R})}:=\left(\int_{\mathbb{R}}\left|f(t)w(t)\right|^pdt\right)^{1/p}.
\end{equation*}
Let $1\le p\le \infty$.
If $\|wf\|_{L_p(\mathbb{R})}<\infty$, then we write $wf\in L_p(\mathbb{R})$,
and here if $p=\infty$, we suppose that $f\in C(\mathbb{R})$ and $\lim_{|x|\rightarrow \infty}|w(x)f(x)|=0$.
We denote the rate of approximation of  $f$ by
\begin{equation*}
E_{p,n}(w;f):=\inf_{P\in\mathcal{P}_n}\left\|(f-P)w\right\|_{L_p(\mathbb{R})}.
\end{equation*}
The Mhaskar-Rakhmanov-Saff numbers $a_x$ is defined as follows:
\begin{equation*}
x=\frac{2}{\pi}\int_0^1\frac{a_xuQ'(a_xu)}{\sqrt{1-u^2}}du,  \quad x>0.
\end{equation*}
To write our theorems we need some preliminaries. We need further assumptions.
\begin{definition}\label{Definition 2.1}
Let $w=\exp(-Q)\in \mathcal{F}(C^2+)$ and  $0< \lambda<(r+2)/(r+1)$. Let $r\ge 1$ be an integer.
Then we write $w\in\mathcal{F}_\lambda(C^{r+2}+)$
if $Q\in C^{(r+2)}(\mathbb{R}\backslash\{0\})$
and there exist two constants $C>1$ and $K\ge 1$ such that for all $|x|\ge K$,
\begin{equation*}
\frac{|Q'(x)|}{Q^{\lambda}(x)} \leq C \quad \mbox{and} \quad
\left| \frac{Q''(x)}{Q'(x)} \right| \sim \left|\frac{Q^{(k+1)}(x)}{Q^{(k)}(x)} \right|
\end{equation*}
for every $k=2,...,r$ and also
\begin{equation*}
\left| \frac{Q^{(r+2)}(x)}{Q^{(r+1)}(x)} \right| \leq C \left|\frac{Q^{(r+1)}(x)}{Q^{(r)}(x)} \right|.
\end{equation*}
In particular, $w\in\mathcal{F}_\lambda(C^{3}+)$ means that $Q\in C^{(3)}(\mathbb{R}\backslash\{0\})$ and
\begin{equation*}
\frac{|Q'(x)|}{Q^{\lambda}(x)} \leq C \quad
\mbox{and}  \quad \left| \frac{Q'''(x)}{Q''(x)} \right| \le C \left|\frac{Q''(x)}{Q'(x)} \right|
\end{equation*}
hold for $|x| \geq K$.
In addition, let $\mathcal{F}_\lambda(C^{2}+):=\mathcal{F}(C^2+)$.
\end{definition}

From {\cite{[5]}}, we know that Example \ref{Example 1.2} (2),(3)
satisfy all conditions of Definition \ref{Definition 2.1}.
Under the same condition of Definition \ref{Definition 2.1} we obtain an interesting theorem as follows:
\begin{theorem}[{\cite[Theorem 4.1, 4.2]{[10]}}]\label{Theorem 2.2}
Let $r$ be a positive integer, $0< \lambda<(r+2)/(r+1)$
and let $w=\exp(-Q)\in \mathcal{F}_\lambda(C^{r+2}+)$.
Then for any $\mu$, $\nu$, $\alpha$, $\beta\in\mathbb{R}$,
we can construct a new weight $w_{\mu,\nu,\alpha,\beta}\in \mathcal{F}_\lambda(C^{r+1}+)$ such that
\begin{equation*}
T_w^{\mu}(x)(1+x^2)^\nu(1+Q(x))^{\alpha}(1+|Q'(x)|)^\beta w(x)\sim w_{\mu,\nu,\alpha,\beta}(x)
\end{equation*}
on $\mathbb{R}$, and
\begin{equation*}
a_{n/c}(w)\le a_n(w_{\mu.\nu,\alpha,\beta})\le a_{cn}(w),  \quad    c\ge 1,
\end{equation*}
\begin{equation*}
T_{w_{\mu,\nu,\alpha,\beta}}(x)\sim T_w(x)
\end{equation*}
hold on $\mathbb{R}$.
\end{theorem}
For a given $\alpha\in \mathbb{R}$ and $w\in\mathcal{F}(C^2+)$,
we let $w_{\alpha}\in\mathcal{F}(C^2+)$ satisfy $w_\alpha(x) \sim T_{w}^{\alpha}(x)w(x)$,
and let $P_{n;f,w_\alpha}\in\mathcal{P}_n$ be the best approximation of $f$ with respect to the weight $w_\alpha$, that is,
\begin{equation*}
\left\|(f-P_{n;f,w_\alpha})w_{\alpha}\right\|_{L_\infty(\mathbb{R})}=E_{n}(w_\alpha,f)
:=\inf_{P\in\mathcal{P}_n}\left\|(f-P)w_\alpha\right\|_{L_\infty(\mathbb{R})}.
\end{equation*}
Then we have the main result as follows:
\begin{theorem}\label{Theorem 2.3}
Let $r \ge 0$ be an integer. Let $w=\exp(-Q)\in \mathcal{F}_\lambda(C^{r+2}+)$ and  $0<  \lambda<(r+2)/(r+1)$.
Suppose that $f\in C^{(r)}(\mathbb{R})$ with
\begin{equation*}
 \lim_{|x|\rightarrow \infty}T^{1/4}(x)f^{(r)}(x)w(x)=0.
\end{equation*}
Then there exists an absolute constant $C_r>0$ which depends only on $r$ such that
for $0\le k\le r$ and $x\in\mathbb{R}$,
\begin{eqnarray*}
\left|\left(f^{(k)}(x)-P_{n;f,w}^{(k)}(x)\right)w(x)\right|
&\le& C_r T^{k/2}(x)E_{n-k}\left(w_{1/4},f^{(k)}\right)\\
&\le& C_r T^{k/2}(x) \left(\frac{a_n}{n}\right)^{r-k}E_{n-r}\left(w_{1/4},f^{(r)}\right).
\end{eqnarray*}
When $w\in\mathcal{F}^*$, we can replace $w_{1/4}$ with $w$ in the above.
\end{theorem}
Applying Theorem \ref{Theorem 2.3} with $w$ or $w_{-1/4}$, we have the following corollary.
\begin{corollary}\label{Corollary 2.4}
{\rm (1)} Let $w=\exp(-Q)\in \mathcal{F}_\lambda(C^{r+2}+),   0<  \lambda<(r+2)/(r+1),   r\ge 0$.
We suppose that $f\in C^{(r)}(\mathbb{R})$ with
\begin{equation*}
 \lim_{|x|\rightarrow \infty}T^{1/4}(x)f^{(r)}(x)w(x)=0,
\end{equation*}
then for $0\le k\le r$ we have
\begin{eqnarray*}
\left\|\left(f^{(k)}-P_{n;f,w}^{(k)}\right)w_{-k/2}\right\|_{L_\infty(\mathbb{R})}
&\le& C_r E_{n-k}\left(w_{1/4},f^{(k)}\right) \\
&\le& C_r \left(\frac{a_n}{n}\right)^{r-k}E_{n-r}\left(w_{1/4},f^{(r)}\right).
\end{eqnarray*}
\item[\,\,\,(2)] Let $w=\exp(-Q)\in \mathcal{F}_\lambda(C^{r+3}+),   0<  \lambda<(r+3)/(r+2),   r\ge 0$.
We suppose that $f\in C^{(r)}(\mathbb{R})$  with
\begin{equation*}
 \lim_{|x|\rightarrow \infty}f^{(r)}(x)w(x)=0,
\end{equation*}
then for $0\le k\le r$ we have
\begin{eqnarray*}
\left\|\left(f^{(k)}-P_{n;f,w_{-1/4}}^{(k)}\right)w_{-(2k+1)/4}\right\|_{L_\infty(\mathbb{R})}
&\le& C_r E_{n-k}\left(w,f^{(k)}\right) \\
&\le& C_r \left(\frac{a_n}{n}\right)^{r-k}E_{n-r}\left(w,f^{(r)}\right).
\end{eqnarray*}
When $w\in\mathcal{F}^*$, we can replace $w_\alpha$ with $w$ in the above.
\end{corollary}
\begin{corollary}\label{Corollary 2.5}
Let $r \ge 0$ be an integer. Let $w=\exp(-Q)\in \mathcal{F}_\lambda(C^{r+3}+)$,   $0< \lambda<(r+3)/(r+2)$,
and let $w_{(2r+1)/4}f^{(r)}\in L_\infty(\mathbb{R})$.
Then, for each $k   (0\le k\le r)$ and the best approximation polynomial $P_{n;f,w_{k/2}}$;
\begin{equation*}
 \left\|\left(f-P_{n;f,w_{k/2}}\right)w_{k/2}\right\|_{L_\infty(\mathbb{R})}=E_{n}\left(w_{k/2},f\right),
\end{equation*}
we have
\begin{eqnarray*}
\left\|\left(f^{(k)}-P_{n;f,w_{k/2}}^{(k)}\right)w\right\|_{L_\infty(\mathbb{R})}
&\le& C_r E_{n-k}\left(w_{(2k+1)/4},f^{(k)}\right)\\
&\le& C_r \left(\frac{a_n}{n}\right)^{r-k}E_{n-r}\left(w_{(2k+1)/4},f^{(r)}\right).
\end{eqnarray*}
When $w\in\mathcal{F}^*$, we can replace $w_{1/4}$ with $w$ in the above.
\end{corollary}

\setcounter{equation}{0}
\section{Proof of Theorems}

      Throughout this section we suppose $w\in \mathcal{F}(C^2+)$.
We give the proofs of theorems. First, we give some lemmas to prove the theorems.
We construct the orthonormal polynomials $p_n(x)=p_n(w^2,x)$ of degree n for $w^2(x)$, that is,
\begin{equation*}
\int_{-\infty}^{\infty} p_n(w^2,x)p_m(w^2,x)w^2(x)dx=\delta_{mn}      (\textrm{Kronecker delta}).
\end{equation*}
Let $fw\in L_2(\mathbb{R})$. The Fourier-type series of $f$ is defined by
\begin{equation*}
\tilde{f}(x):=\sum_{k=0}^{\infty} a_k(w^2,f)p_k(w^2,x),
\quad  a_k(w^2,f):=\int_{-\infty}^{\infty} f(t)p_k(w^2,t)w^2(t)dt.
\end{equation*}
We denote the partial sum of $\tilde{f}(x)$ by
\begin{equation*}
 s_n(f,x):=s_n(w^2,f,x):=\sum_{k=0}^{n-1} a_k(w^2,f)p_k(w^2,x).
\end{equation*}
Moreover, we define the de la Vall$\acute{\textrm{e}}$e Poussin means by
\begin{equation*}
 v_n(f,x):=\frac{1}{n}\sum_{j=n+1}^{2n}s_j(w^2,f,x).
\end{equation*}

\begin{proposition}[{\cite[Theorem 1.1 (1,5), Corollary 6.2 (6.5)]{[11]}}]\label{Proposition 3.1}
Let $w\in\mathcal{F}_\lambda(C^3+),   0<\lambda<3/2$, and let $1\le p\le \infty$.
When $T^{1/4}wf\in L_p(\mathbb{R})$, we have
\begin{equation*}
\left\|v_n(f)w\right\|_{L_p(\mathbb{R})}\le C \left\|T^{1/4}wf \right\|_{L_p(\mathbb{R})},
\end{equation*}
and so
\begin{equation*}
 \left\|(f-v_n(f))w \right\|_{L_p(\mathbb{R})}\le C E_{p,n}\left(T^{1/4}w,f\right).
\end{equation*}
So, equivalently,
\begin{equation*}
\left\|v_n(f)w\right\|_{L_p(\mathbb{R})}\le C \left\|w_{1/4}f\right\|_{L_p(\mathbb{R})},
\end{equation*}
and so
\begin{equation}\label{3.1}
 \left\|(f-v_n(f))w\right\|_{L_p(\mathbb{R})}\le C E_{p,n} \left(w_{1/4},f\right).
\end{equation}
When $w\in\mathcal{F}^*$, we can replace $w_{1/4}$ with $w$.
\end{proposition}
\begin{lemma}\label{Lemma 3.2}
{\rm (1)} {\rm ({\cite[Lemma 3.5, (a)]{[8]}})}
Let $L>0$ be fixed. Then, uniformly for $t>0$,
\begin{equation*}
 a_{Lt}\sim a_t.
\end{equation*}
\item[\,\,\,(2)]  {\rm (\cite[Lemma 3.4 (3.17)]{[8]})}
For $x >1$, we have
\begin{equation*}
 |Q'(a_x)| \sim \frac{x \sqrt{T(a_x)}}{a_x} \quad \mbox{and} \quad |Q(a_x)| \sim \frac{x}{ \sqrt{T(a_x)}}.
\end{equation*}
\item[\,\,\,(3)]  {\rm (\cite[Lemma 3.2 (3.8)]{[8]})}
Let $x\in(0,   \infty)$. There exists $0<\varepsilon<1$ such that
\begin{equation*}
T\left(x\left[1+\frac{\varepsilon}{T(x)}\right]\right)\sim T(x).
\end{equation*}
\item[\,\,\,(4)] {\rm (\cite[Proposition 3]{[9]})}
If $T(x)$   is   unbounded,   then   for   any   $\eta>0$   there   exists   $C(\eta)>0$   such   that   for   $t\ge1$,
\begin{equation*}
 a_t\le C(\eta)t^\eta.
\end{equation*}
\end{lemma}
To prove the results, we need the following notations.
We set
\begin{equation*}
\sigma(t):=\inf \left\{a_u: \,\, \frac{a_u}{u}\le t \right\}, \quad     t>0,
\end{equation*}
and
\begin{equation*}
\Phi_t(x):=\sqrt{\left|1-\frac{|x|}{\sigma(t)}\right|}+T^{-1/2}(\sigma(t)), \quad x\in\mathbb{R}.
\end{equation*}
Define for $fw\in L_p(\mathbb{R})$,  $0<p\le \infty$,
\begin{eqnarray*}
\omega_p(f,w,t)&:=&\sup_{0<h\le t}
\left\|w(x)\left\{f\left(x+\frac{h}{2}\Phi_t(x)\right)-
f\left(x-\frac{h}{2}\Phi_t(x)\right)\right\}\right\|_{L_p(|x|\le \sigma(2t))}\\
&& \quad +\inf_{c\in\mathbb{R}}\left\|w(x)(f-c)(x)\right\|_{L_p(|x|\ge \sigma(4t))}
\end{eqnarray*}
(see \cite{[2],[3]}).
\begin{proposition}[cf.{\cite[Theorem 1.2]{[3]}, \cite[Corollary 1.4]{[2]}}]\label{Proposition 3.3}
Let $w\in \mathcal{F}(C^2+)$.
Let $0<p\le \infty$.
Then for $f:   \mathbb{R}\rightarrow \mathbb{R}$ which $fw\in L_p(\mathbb{R})$
(and for $p=\infty$, we require $f$ to be continuous, and $fw$ to vanish at $\pm \infty$), we have for $n\ge C_3$,
\begin{equation*}
 E_{p,n}\left(f,w\right)\le C_1\omega_{p}\left(f,w,C_2\frac{a_n}{n}\right),
\end{equation*}
where $C_j$,   $j=1,2,3$, do not depend on $f$ and $n$.
\end{proposition}
\begin{proof}
Damelin and Lubinsky \cite{[3]} or Damelin \cite{[2]} have treated a certain class $\mathcal{E}_1$ of weights
containing the conditions (a)-(d) in Definition \ref{Definition 1.1} and
\begin{equation}\label{3.2}
\frac{yQ'(y)}{xQ'(x)}\le C_1 \left(\frac{Q(y)}{Q(x)}\right)^{C_2},
\quad  y\ge x\ge C_3,
\end{equation}
where $C_i$,   $i=1,2,3>0$ are some constants, and they obtain this Proposition
for $w\in \mathcal{E}_1$. Therefore, we may show $\mathcal{F}(C^2+)\subset\mathcal{E}_1$.
In fact, from Definition \ref{Definition 1.1} (d) and (e), we have for $y\ge x>0$,
\begin{equation*}
 \frac{Q'(y)}{Q'(x)}=\exp \left(\int_x^y\frac{Q''(t)}{Q'(t)}dt\right)
 \le\exp \left(C_3\int_x^y\frac{Q'(t)}{Q(t)}dt\right)=\left(\frac{Q(y)}{Q(x)}\right)^{C_3},
\end{equation*}
and
\begin{equation*}
\frac{y}{x}=\exp \left(\int_x^y\frac{1}{t}dt\right)
\le \exp \left(\frac{1}{\Lambda}\int_x^y\frac{Q'(t)}{Q(t)}dt\right)
=\left(\frac{Q(y)}{Q(x)}\right)^{\frac{1}{\Lambda}}.
\end{equation*}
Therefore, we obtain (\ref{3.2}) with $C_2=C_3+\frac{1}{\Lambda}$,
that is, we see $\mathcal{F}(C^2+)\subset\mathcal{E}_1$.
\end{proof}

\begin{theorem}\label{Theorem 3.4}
Let $w\in \mathcal{F}(C^2+)$.
{\rm (1)} If $f$ is a function having bounded variation on any compact interval and if\begin{equation*}
 \int_{-\infty}^{\infty} w(x)|df(x)|<\infty,
\end{equation*}
then there exists a constant $C>0$ such that for every $t>0$,
\begin{equation*}
\omega_1(f,w,t)\le C t\int_{-\infty}^{\infty} w(x)|df(x)|,
\end{equation*}
and so
\begin{equation*}
 E_{1,n}(f,w)\le C\frac{a_n}{n}\int_{-\infty}^{\infty} w(x)|df(x)|.
\end{equation*}
\item[\,\,\,(2)] Let us suppose that $f$ is continuous and $\lim_{|x|\rightarrow \infty}|(\sqrt{T}wf)(x)|=0$, then we have
\begin{equation*}
 \lim_{t\rightarrow 0}\omega_{\infty}(f,w,t)=0.
\end{equation*}
\end{theorem}

To prove this theorem we need the following lemma.
\begin{lemma}[{\cite[Lemma 7]{[9]}}]\label{Lemma 3.5}
{\rm (1)} For $t>0$ there exists $a_u$ such that
\begin{equation*}
 t=\frac{a_u}{u}  \quad  \textrm{and} \quad   \sigma(t)=a_u.
\end{equation*}
\item[\,\,\,(2)] If $t=a_u/u$, $u>0$ large enough and
\begin{equation*}
 |x-y|\le t\Phi_t(x),
\end{equation*}
then there exist $C_1,C_2>0$ such that
\begin{equation*}
 C_1w(x)\le w(y)\le C_2 w(x).
\end{equation*}
\end{lemma}
\begin{proof}[Proof of Theorem \ref{Theorem 3.4}]
(1) Let $g(x):=f(x)-f(0)$. For $t>0$ small enough let $0<h\le t$ and  $|x|\le \sigma(2t)<\sigma(t)$.
Hence we may consider $\Phi_t(x)\le 2$.
Then by Lemma \ref{Lemma 3.5},
\begin{eqnarray*}
&&\int_{-\infty}^{\infty} w(x)\left|g\left(x+\frac{h}{2}\Phi_t(x)\right)-g\left(x-\frac{h}{2}\Phi_t(x)\right)\right|dx\\
&&=\int_{-\infty}^{\infty} w(x)\left|\int_{x-\frac{h}{2}\Phi_t(x)}^{x+\frac{h}{2}\Phi_t(x)}df(v)\right| dx
\le C\int_{-\infty}^{\infty} \left|\int_{x-\frac{h}{2}\Phi_t(x)}^{x+\frac{h}{2}\Phi_t(x)}w(v)df(v)\right| dx\\
&&\le\int_{-\infty}^{\infty} \int_{x-h}^{x+h} w(v)|df(v)|dx
\le\int_{-\infty}^{\infty} w(v)\int_{v-h\le x\le v+h} dx|df(v)|\\
&&\le 2h\int_{-\infty}^{\infty} w(v)|df(v)|.
\end{eqnarray*}
Hence we have
\begin{equation}\label{3.3}
\int_{-\infty}^{\infty} w(x)\left|g\left(x+\frac{h}{2}\Phi_t(x)\right)
-g\left(x-\frac{h}{2}\Phi_t(x)\right)\right| dx
\le 2t \int_{-\infty}^{\infty} w(x)|df(x)|.
\end{equation}
Moreover, we see
\begin{equation}\label{3.4}
\inf_{c\in\mathbb{R}}\left\|w(x)(f-c)(x)\right\|_{L_1(|x|\ge \sigma(4t))}
\le \frac{1}{Q'(\sigma(4t))}\left\|Q'(x)w(x)g(x)\right\|_{L_1(|x|\ge \sigma(4t))}.
\end{equation}
Here we see
\begin{equation}\label{3.5}
 \frac{\sqrt{T(\sigma(t))}}{Q'(\sigma(t))}\sim t.
\end{equation}
In fact, from Lemma \ref{Lemma 3.2} (2), for $t=\frac{a_u}{u}$
\begin{equation*}
 Q'(\sigma(t))= Q'(a_u)\sim \frac{u\sqrt{T(a_u)}}{a_u}\sim \frac{\sqrt{T(\sigma(t))}}{t}.
\end{equation*}
On the other hand, we have
\begin{eqnarray*}
\int_{0}^{\infty} Q'(x)w(x)|g(x)|dx
&=& \int_{0}^{\infty} Q'(x)w(x)\left|\int_0^xdg(u)\right|dx\\
&\le& \int_{0}^{\infty} Q'(x)w(x)\int_0^x|df(u)|dx\\
&=& -w(x)\int_0^x|df(u)|\bigg|_0^{\infty}
+\int_0^{\infty} w(x)|df(x)|\\
&=&\int_0^{\infty} w(x)|df(x)|
\end{eqnarray*}
because $\lim_{x\to \infty} w(x) =0$ from (c) of Definition \ref{Definition 1.1}.
Similarly we have
\begin{equation*}
 \int_{-\infty}^0 \left| Q'(x)w(x)g(x)\right|dx
 \le \int_{-\infty}^0 w(x)|df(x)|.
\end{equation*}
Hence we have
\begin{equation}\label{3.6}
 \| Q'wg\|_{L_1(\mathbb{R})}\le \int_{-\infty}^{\infty} w(u)|df(u)|.
\end{equation}
Therefore, using (\ref{3.4}), (\ref{3.5}) and (\ref{3.6}), we have
\begin{equation}\label{3.7}
\inf_{c\in\mathbb{R}}\left\|w(x)(f-c)(x)\right\|_{L_1(|x|\ge \sigma(4t))}
= O(t) \int_{-\infty}^{\infty} w(x)|df(x)|.
\end{equation}
Consequently, by (\ref{3.3}) and (\ref{3.7}) we have
\begin{equation*}
 \omega_1(f,w,t)\le C t\int_{-\infty}^{\infty} w(x)|df(x)|.
\end{equation*}
Hence, setting  $t=C_2\frac{a_n}{n}$, if we use Proposition \ref{Proposition 3.3},
then
\begin{equation*}
 E_{1,n}(f,w)\le C\frac{a_n}{n}\int_{-\infty}^{\infty} w(x)|df(x)|.
\end{equation*}
(2) Given $\varepsilon>0$, and let us take $L=L(\varepsilon)>0$ large enough as
\begin{equation*}
\sup_{|x|\ge L}|w(x)f(x)| \le \frac{1}{\sqrt{T(L)}}\sup_{|x|\ge L}|\sqrt{T(x)}w(x)f(x)|<\varepsilon
\quad (\textrm{by our assumption}).
\end{equation*}
Then we have
\begin{equation*}
\inf_{c\in\mathbb{R}}\sup_{|x|\ge L} \left|w(x)(f-c)(x)\right|
\le \frac{1}{\sqrt{T(L)}}\sup_{|x|\ge L}\left|\sqrt{T(x)}w(x)f(x)\right|<\varepsilon.
\end{equation*}
Now, there exists $\varepsilon >0$ small enough such that
\begin{equation*}
 \frac{h}{2}\Phi_t(x)\le \varepsilon \frac{1}{T(x)},  \quad |x|\le \sigma(2t),
\end{equation*}
because if we put $t=a_u/u$, then we see $\sigma(t)=a_u$ and $|x|\le \sigma(2t)<a_u$.
Hence, noting \cite[Lemma 3.7]{[8]}, that is,
for some $\varepsilon >0$, and for large enough $t$,
\begin{equation*}
T(a_t) \le C t^{2-\varepsilon},
\end{equation*}
and if $w$ is the Erd\"{o}s-type weight, then from Lemma \ref{Lemma 3.2} (4), we have
\begin{equation}\label{3.8}
 t\Phi_t(x)\le \frac{a_u}{u}\le \varepsilon \frac{1}{T(a_u)}\le \varepsilon \frac{1}{T(x)}.
\end{equation}
If $w\in\mathcal{F}^*$, we also have (\ref{3.8}), because for some $\delta >0$ and $u>0$ large enough,
\begin{equation*}
 t\Phi_t(x)\le \frac{a_u}{u}\le u^{-\delta} \le \varepsilon \frac{1}{T(x)}.
\end{equation*}
Therefore, using Lemma \ref{Lemma 3.2} (3), Lemma \ref{Lemma 3.5}
and the assumption
\[
\lim_{|x|\to \infty}\sqrt{T\left(x\right)}w\left(x\right)f\left(x\right)=0,
\]
for $2L\le |x|\le \sigma(2t)$,   $h>0$,
\begin{eqnarray*}
&&\left|w(x)\left\{f\left(x+\frac{h}{2}\Phi_t(x)\right)-f\left(x-\frac{h}{2}\Phi_t(x)\right)\right\}\right|\\
&\le& C\Bigg[\frac{1}{\sqrt{T(x)}}
\left|\sqrt{T\left(x+\frac{h}{2}\Phi_t(x)\right)}w\left(x+\frac{h}{2}\Phi_t(x)\right)f\left(x+\frac{h}{2}\Phi_t(x)\right)\right|\\
&& \qquad +\frac{1}{\sqrt{T(x)}}
\left|\sqrt{T\left(x-\frac{h}{2}\Phi_t(x)\right)}w\left(x-\frac{h}{2}\Phi_t(x)\right)f\left(x-\frac{h}{2}\Phi_t(x)\right)\right|\Bigg]\\
&\le& 2\varepsilon.
\end{eqnarray*}
On the other hand,
\begin{equation*}
\lim_{t\rightarrow 0}\sup_{0<h\le t}
\left\|w(x) \left\{f\left(x+\frac{h}{2}\Phi_t(x)\right)
-f\left(x-\frac{h}{2}\Phi_t(x)\right)\right\}\right\|_{L_\infty(|x|\le 2L)}=0.
\end{equation*}
Therefore, we have the result.
\end{proof}
\begin{lemma}[cf.{\cite[Lemma 4.4]{[4]}}]\label{Lemma 3.6}
Let $g$ be a real valued function on $\mathbb{R}$ satisfying $\|gw\|_{L_\infty(\mathbb{R})}<\infty$ and
\begin{equation}\label{3.9}
 \int_{-\infty}^{\infty} gPw^2dt=0 \quad P\in\mathcal{P}_n.
\end{equation}
Then we have
\begin{equation}\label{3.10}
\left\|w(x)\int_0^x g(t)dt \right\|_{L_\infty(\mathbb{R})}
\le C \frac{a_n}{n}\|gw\|_{L_\infty(\mathbb{R})}.
\end{equation}
Especially, if $w\in \mathcal{F}_\lambda(C^3+),   0<\lambda<3/2$, then we have
\begin{equation}\label{3.11}
\left\|w(x)\int_0^x \left(f'(t)-v_n(f')(t)\right)dt \right\|_{L_\infty(\mathbb{R})}
\le C \frac{a_n}{n}E_{n} \left(w_{1/4},f'\right).
\end{equation}
When $w\in \mathcal{F}^*$, we also have (\ref{3.11}) replacing $w_{1/4}$ with $w$.
\end{lemma}
\begin{proof}
We let
\begin{equation}\label{3.12}
\phi_x(t)=
\left\{
\begin{array}{lr}
         w^{-2}(t),& 0\le t\le x; \\
         0,& otherwise,
\end{array}
         \right.
\end{equation}
then we have for arbitrary $P_n\in\mathcal{P}_n$,
\begin{equation}\label{3.13}
\left|\int_0^x g(t)dt\right|
=\left|\int_{-\infty}^{\infty} g(t)\phi_x(t)w^2(t)dt\right|
=\left|\int_{-\infty}^{\infty} g(t)(\phi_x(t)-P_n(t))w^2(t)dt\right|.
\end{equation}
Therefore, we have
\begin{eqnarray*}
\left|\int_0^x g(t)dt\right|&\le& \|gw\|_{L_\infty(\mathbb{R})}
\inf_{P_n\in\mathcal{P}_n}\int_{-\infty}^{\infty} \left|\phi_x(t)-P_n(t)\right|w(t)dt \\
&=&\|gw\|_{L_\infty(\mathbb{R})}E_{1,n}(w:\phi_x).
\end{eqnarray*}
Here, from Theorem \ref{Theorem 3.4} we see that
\begin{eqnarray*}
E_{1,n}(w:\phi_x)&\le& C\frac{a_n}{n}\int_{-\infty}^{\infty} w(t)|d\phi_x(t)|
\le C\frac{a_n}{n}\int_0^x w(t)|Q'(t)|w^{-2}(t)dt\\
&=& C\frac{a_n}{n}\int_0^x Q'(t)w^{-1}(t)dt
\le C\frac{a_n}{n}w^{-1}(x).
\end{eqnarray*}
So, we have
\begin{equation*}
\left|w(x)\int_0^x g(t)dt\right|\le
\|gw\|_{L_\infty(\mathbb{R})}w(x)E_{1,n}\left(w:\phi_x\right)
\le C\frac{a_n}{n}\|gw\|_{L_\infty(\mathbb{R})}.
\end{equation*}
Therefore, we have (\ref{3.10}). Next we show (\ref{3.11}).
Since
\begin{equation*}
v_n(f')(t)=\frac{1}{n}\sum_{j=n+1}^{2n}s_j(f',t),
\end{equation*}
and for any $P\in\mathcal{P}_n$,   $j\ge n+1$,
\begin{equation*}
 \int_{-\infty}^{\infty} \left(f'(t)-s_j(f';t)\right)P(t)w^2(t)dt=0,
\end{equation*}
we have
\begin{equation*}
 \int_{-\infty}^{\infty} \left(f'(t)-v_n(f')(t)\right)P(t)w^2(t)dt=0.
\end{equation*}
Using (\ref{3.10}) and (\ref{3.1}), we have (\ref{3.11}).
\end{proof}

\begin{lemma}\label{Lemma 3.7}
Let $w=\exp(-Q)\in\mathcal{F}_\lambda(C^3+)$,   $0<\lambda<3/2$.
Let $\left\|w_{1/4}f'\right\|_{L_\infty(\mathbb{R})}<\infty$,
and let $q_{n-1}\in \mathcal{P}_n$ be the best approximation of $f'$ with respect to the weight $w$,
that is,
\begin{equation*}
\left\|(f'-q_{n-1})w \right\|_{L_\infty(\mathbb{R})}=E_{n-1}(w,f').
\end{equation*}
Now we set
\begin{equation*}
 F(x):=f(x)-\int_0^xq_{n-1}(t)dt,
\end{equation*}
then there exists $S_{2n}\in \mathcal{P}_{2n}$ such that
\begin{equation*}
 \left\|w \left(F-S_{2n}\right) \right\|_{L_\infty(\mathbb{R})}
 \le C  \frac{a_n}{n}E_n \left(w_{1/4},f'\right),
\end{equation*}
and
\begin{equation*}
 \left\|wS_{2n}' \right\|_{L_\infty(\mathbb{R})}\le C E_{n-1}\left(w_{1/4},f'\right).
\end{equation*}
When $w\in \mathcal{F}^*$, we also have same results replacing $w_{1/4}$ with $w$.
\end{lemma}
\begin{proof}
Let
\begin{equation}\label{3.14}
 S_{2n}(x)=f(0)+\int_0^x v_n \left(f'-q_{n-1}\right)(t)dt,
\end{equation}
then by Lemma \ref{Lemma 3.6} (\ref{3.11}),
\begin{eqnarray*}
&&\left\|w\left(F-S_{2n}\right)\right\|_{L_\infty(\mathbb{R})} \\
&=&\left\|w\left(f-\int_0^xq_{n-1}(t)dt -f(0)-\int_0^x v_n\left(f'-q_{n-1}\right)(t)dt\right)\right\|_{L_\infty(\mathbb{R})}\\
&=&\left\|w\left(\int_0^x \left[f'(t)-v_n(f')(t)\right]dt\right)\right\|_{L_\infty(\mathbb{R})}
\le  C\frac{a_n}{n}E_n\left(w_{1/4},f'\right).
\end{eqnarray*}
Now by Proposition \ref{Proposition 3.1} (\ref{3.1}),
\begin{eqnarray*}
\left\|wS_{2n}'\right\|_{L_\infty(\mathbb{R})}
&=&\left\|w\left(v_n(f'-q_{n-1})\right)\right\|_{L_\infty(\mathbb{R})}\\
&\le&\left\|\left(f'-v_n\left(f'\right)\right)w\right\|_{L_\infty(\mathbb{R})}
+\left\|\left(f'-q_{n-1}\right)w\right\|_{L_\infty(\mathbb{R})}\\
&\le& E_{n}\left(w_{1/4},f'\right)+E_{n-1}\left(w,f'\right)
\le E_{n-1}\left(w_{1/4},f'\right).
\end{eqnarray*}
\end{proof}

To prove Theorem \ref{Theorem 2.3} we need the following theorems.

\begin{theorem}[{\cite[Corollary 3.4]{[9]}}] \label{Theorem 3.8}
Let $w\in\mathcal{F}(C^2+)$, and let $r\ge 0$ be an integer.
Let $1\le p\le \infty$, and let $wf^{(r)}\in L_p(\mathbb{R})$. Then we have
\begin{equation*}
 E_{p,n}(f,w)\le C  \left(\frac{a_n}{n}\right)^k \left\|f^{(k)}w\right\|_{L_p(\mathbb{R})}, \quad   k=1,2,...,r,
\end{equation*}
and equivalently,
\begin{equation*}
 E_{p,n}(f,w)\le C  \left(\frac{a_n}{n}\right)^kE_{p,n-k}\left(f^{(k)},w\right).
\end{equation*}
\end{theorem}
\begin{theorem}[{\cite[Corollary 6.2]{[10]}}]\label{Theorem 3.9}
Let $r\ge 1$ be an integer and $w\in \mathcal{F}_\lambda(C^{r+2}+)$,
$0< \lambda<(r+2)/(r+1)$, and let $1\le p\le \infty$.
Then there exists a constant $C>0$ such that
for any $1\le k\le r$, any integer $n\ge 1$ and any polynomial $P\in\mathcal{P}_n$,
\begin{equation*}
 \left\|P^{(k)}w\right\|_{L_p(\mathbb{R})}
 \le C\left(\frac{n}{a_n}\right)^k\left\|T^{k/2}Pw\right\|_{L_p(\mathbb{R})}.
\end{equation*}
\end{theorem}
\begin{proof}[Proof of Theorem \ref{Theorem 2.3}]
We prove the theorem only in case of unbounded $T(x)$,
in the case of Freud case $\mathcal{F}^*$ we can prove it similarly.
We show that for $k=0,1,...,r$,
\begin{equation}\label{3.15}
\left|\left(f^{(k)}(x)-P_{n;f,w}^{(k)}\right)w(x)\right|
\le CT^{k/2}(x)E_{n-k}\left(w_{1/4},f^{(k)}\right).
\end{equation}
If $r=0$, then (\ref{3.15}) is trivial. For some $r\ge 0$ we suppose that (\ref{3.15}) holds,
and let $f\in C^{(r+1)}(\mathbb{R})$. Then $f'\in C^{(r)}(\mathbb{R})$.
Let $q_{n-1}\in\mathcal{P}_{n-1}$ be the polynomial of best approximation of $f'$
with respect to the weight $w$.
Then, from our assumption we have for $0\le k\le r$,
\begin{equation*}
\left|\left(f^{(k+1)}(x)-q_{n-1}^{(k)}(x)\right)w(x)\right|
\le C T^{k/2}(x)E_{n-k}\left(w_{1/4},f^{(k+1)}\right),
\end{equation*}
that is, for $1\le k\le r+1$
\begin{equation}\label{3.16}
\left|\left(f^{(k)}(x)-q_{n-1}^{(k-1)}(x)\right)w(x)\right|
\le C T^{\frac{k-1}{2}}(x)E_{n-k+1}\left(w_{1/4},f^{(k)}\right).
\end{equation}
Let
\begin{equation}\label{3.17}
 F(x):=f(x)-\int_0^x q_{n-1}(t)dt=f(x)-Q_n(x),
\end{equation}
then
\begin{equation*}
 |F'(x)w(x)|\le C E_{n-1}\left(w,f'\right).
\end{equation*}
As (\ref{3.14}) we set $S_{2n}=\int_0^x(v_n(f')(t)-q_{n-1}(t))dt+f(0)$, then from Lemma \ref{Lemma 3.7}
\begin{equation}\label{3.18}
\left\|\left(F-S_{2n}\right)w\right\|_{L_\infty(\mathbb{R})}
\le C  \frac{a_n}{n}E_n\left(w_{1/4},f'\right),
\end{equation}
and
\begin{equation*}
\left\|S_{2n}'w\right\|_{L_\infty(\mathbb{R})}\le C E_{n-1}\left(w_{1/4},f'\right).
\end{equation*}
Here we apply Theorem \ref{Theorem 3.9} with the weight $w_{-(k-1)/2}$.
In fact, by Theorem \ref{Theorem 2.2} we have $w_{-(k-1)/2}\in\mathcal{F}_\lambda(C^{r+2}+)$.
Then, noting $a_{2n}\sim a_n$ from Lemma \ref{Lemma 3.2} (1), we see
\begin{eqnarray*}
|S_{2n}^{(k)}(x) w_{-(k-1)/2}(x))|&\le& C \left(\frac{n}{a_n}\right)^{k-1}\|S_{2n}'w\|_{L_\infty(\mathbb{R})}\\
&\le& C \left(\frac{n}{a_n}\right)^{k-1}E_{n-1}\left(w_{1/4},f'\right),
\end{eqnarray*}
that is,
\begin{equation}\label{3.19}
\left|S_{2n}^{(k)}(x) w(x)\right|
\le C \left(\frac{n\sqrt{T(x)}}{a_n}\right)^{k-1}E_{n-1}\left(w_{1/4},f'\right),
\quad 1\le k\le r+1.
\end{equation}
Let $R_{n}\in \mathcal{P}_{n}$ denote the polynomial of best approximation of  $F$ with $w$.
By Theorem \ref{Theorem 3.9} with $w_{-\frac{k}{2}}$ again, for $0\le k\le r+1$ we have
\begin{eqnarray}\label{3.20} \nonumber
\left|(R_{n}^{(k)}-S_{2n}^{(k)}(x)) w_{-\frac{k}{2}}(x)\right|
&\le& C \left(\frac{n}{a_n}\right)^k\|(R_{n}-S_{2n})w_{-\frac{k}{2}}(x)T^{k/2}(x)\|_{L_\infty(\mathbb{R})}\\
&\le& C \left(\frac{n}{a_n}\right)^k\|(R_{n}-S_{2n})w\|_{L_\infty(\mathbb{R})}
\end{eqnarray}
and by (\ref{3.18})
\begin{eqnarray}\label{3.21} \nonumber
\|(R_{n}-S_{2n})w\|_{L_\infty(\mathbb{R})}
&\le& C \left[\|(F-R_{n})w\|_{L_\infty(\mathbb{R})}+\|(F-S_{2n})w\|_{L_\infty(\mathbb{R})}\right]\\\nonumber
&\le& C \left[E_{n}(w,F)+\frac{a_n}{n}E_{n}\left(w_{1/4},f'\right) \right] \\ \nonumber
&\le& C \left[\frac{a_n}{n}E_{n-1}(w,f')+\frac{a_n}{n}E_{n-1}(w_{1/4},f')\right]\\
&\le& C \frac{a_n}{n}E_{n-1}\left(w_{1/4},f'\right).
\end{eqnarray}
Hence,  from (\ref{3.20}) and (\ref{3.21}) we have for $0\le k\le r+1$
\begin{eqnarray}\label{3.22}\nonumber
|(R_{n}^{(k)}-S_{2n}^{(k)}(x)) w(x)|
&\le& C\left|T^{k/2}(x)\right|\left|(R_{n}^{(k)}-S_{2n}^{(k)}(x)) w_{-\frac{k}{2}}(x)\right|\\
&\le& C \left(\frac{n\sqrt{T(x)}}{a_n}\right)^k\frac{a_n}{n}E_{n-1}\left(w_{1/4},f'\right).
\end{eqnarray}
Therefore by (\ref{3.19}), (\ref{3.22}) and Theorem \ref{Theorem 3.8},
\begin{eqnarray}\label{3.23} \nonumber
|R_{n}^{(k)}(x) w(x))|
&\le& C T^{k/2}(x)\left(\frac{n}{a_n}\right)^{k-1}
E_{n-1}\left(w_{1/4},f'\right) \\
&\le& C T^{k/2}(x)E_{n-k}\left(w_{1/4},f^{(k)}\right).
\end{eqnarray}
Since $E_{n}(F,w)=E_{n}(w,f)$ and
\begin{equation}\label{3.24}
 E_{n}\left(F,w\right)=\left\|w\left(F-R_{n}\right)\right\|_{L_\infty(\mathbb{R})}
 =\left\|w\left(f-Q_n-R_{n}\right)\right\|_{L_\infty(\mathbb{R})}
\end{equation}
(see (\ref{3.17})), we know that $P_{n;f,w}:=Q_n+R_n$ is the polynomial of best approximation of $f$ with $w$.
Now, from (\ref{3.16}), (\ref{3.17}) and (\ref{3.23}) we have for $1\le k\le r+1$,
\begin{eqnarray*}
\left|\left(f^{(k)}(x)-P_{n;f.w}^{(k)}(x)\right)w(x)\right|
&=&\left|\left(f^{(k)}(x)-Q_n^{(k)}(x)-R_{n}^{(k)}(x)\right)w(x)\right|\\
&\le& \left|(f^{(k)}(x)-q_{n-1}^{(k-1)}(x))w(x)\right|
+\left|R_{n}^{(k)}(x)w(x)\right|\\
&\le& C T^{k/2}(x)E_{n-k}\left(w_{1/4},f^{(k)}\right).
\end{eqnarray*}
 For $k=0$ it is trivial. Consequently, we have (\ref{3.15}) for all $r\ge 0$.
Moreover, using Theorem \ref{Theorem 3.8}, we conclude Theorem \ref{Theorem 2.3}.
\end{proof}

\begin{proof}[Proof of Corollary \ref{Corollary 2.4}]
It follows from Theroem \ref{Theorem 2.3}.
\end{proof}

\begin{proof}[Proof of Corollary \ref{Corollary 2.5}]
Applying Theorem \ref{Theorem 2.3} with $w_{k/2}$,
we have for $0\le j\le r$
\begin{equation*}
\left\|(f^{(j)}-P_{n;f,w_{k/2}}^{(j)})w\right\|_{L_\infty(\mathbb{R})}
\le C E_{n-k}\left(w_{(2k+1)/4},f^{(j)}\right).
\end{equation*}
Especially, when $j=k$, we obtain
\begin{equation*}
 \left\|\left(f^{(k)}-P_{n;f,w_{k/2}}^{(k)}\right)w\right\|_{L_\infty(\mathbb{R})}
 \le CE_{p,n-k}\left(w_{(2k+1)/4},f^{(k)}\right).
\end{equation*}
\end{proof}

\setcounter{equation}{0}
\section{Theorems in $L_p(\mathbb{R})$ $(1 \le p \le \infty)$}

In this section we will give an analogy of Theorem \ref{Theorem 2.3}
in $L_p(\mathbb{R})$-space ($1 \le p \le \infty$)
and we will prove it using the same method as the proof of Theorem \ref{Theorem 2.3}.
Let $1 \le p \le \infty$.
Let $w=\exp(-Q)\in\mathcal{F}_\lambda(C^3+)$,   $0<\lambda<3/2$, and let $\beta>1$ be fixed. Then we set
$w^{\sharp}$ and $w^{\flat}$ as follows;
\begin{eqnarray*}
&& \frac{w(x)}{\left\{(1+|Q'(x)|)(1+|x|)^{\beta}\right\}^{1/p}}\sim w^{\sharp}(x)\in\mathcal{F}(C^2+);\\
&&w(x)\left\{(1+|Q'(x)|)(1+|x|)^{\beta}\right\}^{1/p} \sim w^{\flat}(x)\in\mathcal{F}(C^2+)
\end{eqnarray*}
(see Theorem \ref{Theorem 2.2}).
\begin{theorem}\label{Theorem 4.1}
Let $r \ge 0$ be an integer. Let $w=\exp(-Q)\in \mathcal{F}_\lambda(C^{r+2}+)$,
$0< \lambda<(r+2)/(r+1)$,
and let $\beta>1$ be fixed.
Suppose that $T^{1/4}f^{(r)}w\in L_p(\mathbb{R})$.
Let $P_{p,n;f,w}\in\mathcal{P}_n$ be the best approximation of $f$
with respect to the weight $w$ in $L_p(\mathbb{R})$-space, that is,
\begin{equation*}
E_{p,n}\left(w,f\right):=\inf_{P\in\mathcal{P}_n}\left\|\left(f-P\right)w\right\|_{L_p(\mathbb{R})}
=\left\|\left(f-P_{p,n;f,w}\right)w\right\|_{L_p(\mathbb{R})}.
\end{equation*}
Then there exists an absolute constant $C_r>0$ which depends only on $r$
such that for $0\le k\le r$ and $x\in\mathbb{R}$,
\begin{eqnarray*}
\left\|\left(f^{(k)}-P_{p,n;f,w}^{(k)}\right)w^{\sharp}_{-k/2}\right\|_{L_p(\mathbb{R})}
&\le& C_r E_{p,n-k}\left(w_{1/4},f^{(k)}\right)\\
&\le& C_r \left(\frac{a_n}{n}\right)^{r-k}E_{p,n-r}\left(w_{1/4},f^{(r)}\right).
\end{eqnarray*}
When $w\in\mathcal{F}^*$, we can replace $w_{1/4}$
and $w^{\sharp}_{-k/2}$ with $w$ and $w^{\sharp}$, respectively in the above.
\end{theorem}
If we apply Theorem \ref{Theorem 4.1} with $w_{-1/4}$, then we have the following.
\begin{corollary}\label{Corollary 4.2}
Let $r \ge 0$ be an integer. Let $w=\exp(-Q)\in \mathcal{F}_\lambda(C^{r+3}+)$,   $0< \lambda<(r+3)/(r+2)$,
and let $\beta>0$ be fixed.
Suppose that $wf^{(r)}\in L_p(\mathbb{R})$. Then for $0\le k\le r$ we have
\begin{eqnarray*}
\left\|\left(f^{(k)}-P_{p,n;f,w_{-1/4}}^{(k)}\right)w^{\sharp}_{-(2k+1)/4}
\right\|_{L_p(\mathbb{R})}
&\le& C_r E_{p,n-k}\left(w,f^{(k)}\right)\\
&\le& C_r \left(\frac{a_n}{n}\right)^{r-k}E_{p,n-r}\left(w,f^{(r)}\right).
\end{eqnarray*}
When $w\in\mathcal{F}^*$, we can omit $T^{-(2k+1)/4}$ in the above.
\end{corollary}

\begin{corollary}\label{Corollary 4.3}
Let $r \ge 0$ be an integer. Let $w=\exp(-Q)\in \mathcal{F}_\lambda(C^{r+3}+)$,
$0< \lambda<(r+3)/(r+2)$,
and let $\beta>0$ be fixed.
\item[\,\,\,(1)] Let $w_{(2r+1)/4}f^{(r)}\in L_p(\mathbb{R})$.
Then, for each $k$   $(0\le k\le r)$ and  the best approximation polynomial $P_{p,n;f,w_{k/2}}$, we have
\begin{eqnarray*}
\left\|\left(f^{(k)}-P_{p,n;f,w_{k/2}}^{(k)}\right)w^{\sharp}\right\|_{L_p(\mathbb{R})}
&\le& C_r E_{p,n-k}\left(w_{(2k+1)/4},f^{(k)}\right)\\
&\le& C_r \left(\frac{a_n}{n}\right)^{r-k}E_{p,n-r}\left(w_{(2k+1)/4},f^{(r)}\right).
\end{eqnarray*}
\item[\,\,\,(2)] Let $w^{\flat}_{(2r+1)/4}f^{(r)}\in L_p(\mathbb{R})$.
Then, for each $k$ $(0\le k\le r)$ and
 the best approximation polynomial $P_{p,n;f,w^{\flat}_{k/2}}$, we have
\begin{eqnarray*}
\left\|\left(f^{(k)}-P_{p,n;f,w^{\flat}_{k/2}}^{(k)}\right)w\right\|_{L_p(\mathbb{R})}
&\le& C_r E_{p,n-k}\left(w^{\flat}_{(2k+1)/4},f^{(k)}\right)\\
&\le& C_r \left(\frac{a_n}{n}\right)^{r-k}E_{p,n-r}\left(w^{\flat}_{(2k+1)/4},f^{(r)}\right).
\end{eqnarray*}
When $w\in\mathcal{F}^*$, we can replace $w_{(2k+1)/4}$ and $w^{\flat}_{(2k+1)/4}$
with $w$ and $w^{\flat}$, respectively in the above.
\end{corollary}
Especially when $p=\infty$, we can refer to  $w^{\sharp}$ or $w^{\flat}$ as $w$.
In this case, we can note that Corollary \ref{Corollary 4.2} and Corollary \ref{Corollary 4.3}
imply Corollary \ref{Corollary 2.4}, and Corollary \ref{Corollary 2.5}, respectively.

To prove Theorem \ref{Theorem 4.1} we need to prepare some notations and lemmas.
\begin{lemma}[{\cite[Lemma 3.8]{[6]}}] \label{Lemma 4.4}
Let $w\in \mathcal{F}(C^2+)$ and let $1\le p\le \infty$. If $g:\mathbb{R}\rightarrow \mathbb{R}$ is absolutely continuous,
$g(0)=0$, and $wg'\in L_p(\mathbb{R})$, then
\begin{equation*}
 \left\|Q'wg\right\|_{L_p(\mathbb{R})}\le C\left\|wg'\right\|_{L_p(\mathbb{R})}.
\end{equation*}
\end{lemma}
\begin{lemma}[cf.{\cite[Theorem 4.1]{[9]}}]\label{Lemma 4.5}
Let $w\in \mathcal{F}(C^2+)$ and let $1\le p\le \infty$. If $wf'\in L_p(\mathbb{R})$, then
\begin{equation*}
 E_{p,n}(w,f)\le C\omega_p\left(f,w,\frac{a_n}{n}\right)
 \le C \frac{a_n}{n}\left\|wf'\right\|_{L_p(\mathbb{R})}.
\end{equation*}
\end{lemma}
\begin{proof}
The first inequality follows from Proposition \ref{Proposition 3.3}.
We show the second inequality.  By \cite[Lemma 7]{[9]} we have
\begin{eqnarray*}
&&\left\|w(x)\left\{f\left(x+\frac{h}{2}\Phi_t(x)\right)-f\left(x-\frac{h}{2}\Phi_t(x)\right)\right\}
\right\|^p_{L_p(|x|\le \sigma(2t))} \\
&=&h^p\int_{\mathbb{R}}|w(x)\Phi_t(x)f'(x_t)|^pdx
\le Ch^p\int_{\mathbb{R}}|w(x)f'(x)|^pdx.
\end{eqnarray*}
Hence we see
\begin{equation}\label{4.1}
\sup_{0<h\le t}\left\|w(x)\left\{f\left(x+\frac{h}{2}\Phi_t(x)\right)
-f\left(x-\frac{h}{2}\Phi_t(x)\right)\right\}\right\|_{L_p(|x|\le \sigma(2t))}
\le Ct\left\|wf'\right\|_{L_p(\mathbb{R})}.
\end{equation}
Now, we estimate
\begin{equation*}
\left\|w(x)(f-c)(x)\right\|_{L_p(|x|\ge \sigma(4t))}.
\end{equation*}
Let $g(x):=f(x)-f(0)$.
\begin{equation*}\label{4.2*}
\inf_{c\in\mathbb{R}}\left\|w(x)\left(f-c\right)(x)\right\|_{L_p(|x|\ge \sigma(4t))}
\le \frac{1}{Q'(\sigma(4t))} \left\|Q'wg\right\|_{L_p(\mathbb{R})}.
\end{equation*}
Then we have from (\ref{3.5}),
\begin{equation}\label{4.2}
\inf_{c\in\mathbb{R}}\left\|w(x)(f-c)(x)\right\|_{L_p(|x|\ge \sigma(4t))}
\le Ct\|Q'wg\|_{L_p(\mathbb{R})}.
\end{equation}
Here, from Lemma \ref{Lemma 4.4} we have
\begin{equation}\label{4.3}
 \|Q'w(f-f(0))\|_{L_p(\mathbb{R})}\le C\|wf'\|_{L_p(\mathbb{R})}.
\end{equation}
From (\ref{4.2}) and (\ref{4.3}) we have
\begin{equation}\label{4.4}
\inf_{c\in\mathbb{R}}\left\|w(x)(f-c)(x)\right\|_{L_p(|x|\ge \sigma(4t))}\le Ct \|wf'\|_{L_p(\mathbb{R})}.
\end{equation}
From (\ref{4.1}) and (\ref{4.4}) we have the result.
\end{proof}
\begin{lemma}[cf.{\cite[Lemma 4.4]{[4]}}]\label{Lemma 4.6}
Let $1 \le p \le \infty$ and $\beta>1$, and let us define $w^{\sharp}$ with $p$, $\beta$.
Let $g$ be a real valued function on $\mathbb{R}$ satisfying $\|gw\|_{L_p(\mathbb{R})}<\infty$ and
(\ref{3.9}),
then we have
\begin{equation}\label{4.5}
\left\|w^{\sharp}(x)\int_0^x g(t)dt\right\|_{L_p(\mathbb{R})}
\le C \frac{a_n}{n}\|gw\|_{L_p(\mathbb{R})}.
\end{equation}
Especially, if $w\in \mathcal{F}_\lambda(C^3+),   0<\lambda<3/2$, then we have
\begin{equation}\label{4.6}
 \left\|w^{\sharp}(x)\int_0^x \left(f'(t)-v_n(f')(t)\right)dt\right\|_{L_p(\mathbb{R})}
 \le C \frac{a_n}{n}E_{p,n}\left(w_{1/4},f'\right).
\end{equation}
When $w\in \mathcal{F}^*$, we also have (\ref{4.6}) replacing $w_{1/4}$ with $w$.
\end{lemma}
\begin{proof}
For arbitrary $P_n\in\mathcal{P}_n$,  we have by (\ref{3.13}) and H\"older inequality
\begin{equation*}
\left|\int_0^x g(t)dt\right|
\le \left\|g(t)w(t)\right\|_{L_p(\mathbb{R})}
E_{q,n}\left(w,\phi_x\right), \quad 1\le p \le \infty,   1/p+1/q=1,
\end{equation*}
where $\phi$ is defined in (\ref{3.12}).
Then, we obtain by Lemma \ref{Lemma 4.5},
\begin{eqnarray*}
\left|\int_0^x g(t)dt\right|
&\le& \left\|g(t)w(t)\right\|_{L_p(\mathbb{R})}
\frac{a_n}{n}
\left(\int_{\mathbb{R}}|w(t)\phi'_x(t)|^qdt\right)^{1/q} \\
&\le& \frac{a_n}{n}\left\|g(t)w(t)\right\|_{L_p(\mathbb{R})}
|Q'(x)|^{1-1/q}\left(\int_0^x Q'(t)w^{-q}(t)dt\right)^{1/q}\\
&\le& C\frac{a_n}{n}\left\|g(t)w(t)\right\|_{L_p(\mathbb{R})}
|Q'(x)|^{1-1/q}w^{-1}(x).
\end{eqnarray*}
Here, for $p=1$ we may consider
\[
\lim_{q\to \infty} \left( \int_0^xQ'(t)w^{-q}(t)dt\right)^{1/q}
=\lim_{q\to \infty} \left( w^{-q}(t)\right)^{1/q}=w^{-1}(x).
\]
Hence, we have
\begin{eqnarray*}
\left\|\frac{w(x)}{\left\{(1+|Q'(x)|)(1+|x|)^{\beta}\right\}^{1/p}}\int_0^x g(t)dt\right\|_{L_p(\mathbb{R})}
&\le& C\left(\frac{a_n}{n}\right)\left\|\left(1+|x|\right)^{-\beta/p}\right\|_{L_p(\mathbb{R})}
\left\|gw\right\|_{L_p(\mathbb{R})}\\
&\le& C\left(\frac{a_n}{n}\right)\left\|gw\right\|_{L_p(\mathbb{R})}.
\end{eqnarray*}
Therefore, we have (\ref{4.5}). From (\ref{4.5}), (\ref{3.9}) and Proposition \ref{Proposition 3.1},
we have (\ref{4.6}).
\end{proof}
\begin{lemma}\label{Lemma 4.7}
Let $w=\exp(-Q)\in\mathcal{F}_\lambda(C^3+)$,   $0<\lambda<3/2$.
Let $1 \le p \le \infty$, $\|w_{1/4}f'\|_{L_p(\mathbb{R})}<\infty$,
and let $q_{n-1}\in \mathcal{P}_n$ be the best approximation of $f'$ with respect to the weight $w$
on $L_p(\mathbb{R})$ space,
that is,
\begin{equation*}
\left\|(f'-q_{n-1})w\right\|_{L_p(\mathbb{R})}=E_{p,n-1}(w,f').
\end{equation*}
Using $q_{n-1}$, define $F(x)$ and $S_{2n}$  as (\ref{3.17}) and (\ref{3.14}).
Then we have
\begin{equation*}
 \left\|w^{\sharp}(F-S_{2n})\right\|_{L_p(\mathbb{R})}\le C  \frac{a_n}{n}E_{p,n}\left(w_{1/4},f'\right),
\end{equation*}
and
\begin{equation*}
 \left\|wS_{2n}'\right\|_{L_p(\mathbb{R})}\le C E_{p,n-1}\left(w_{1/4},f'\right).
\end{equation*}
When $w\in \mathcal{F}^*$, we also have same results replacing $w_{1/4}$ with $w$.
\end{lemma}
\begin{proof}
By Lemma \ref{Lemma 4.6} (\ref{4.6}), we have the result using the same method as the proof of Lemma \ref{Lemma 3.7}.
\end{proof}
\begin{proof}[Proof of Theorem \ref{Theorem 4.1}]
We will prove it similarly to the proof of Theorem \ref{Theorem 2.3}.
First, let $q_{n-1}\in\mathcal{P}_{n-1}$ be the polynomial of best approximation of  $f'$ with respect to the weight $w$
on $L_p(\mathbb{R})$ space.
Then using $q_{n-1}$, we define $F(x)$ and $S_{2n}$ in the same method as (\ref{3.17}) and (\ref{3.14}).
Then we have using Lemma \ref{Lemma 4.7}
\begin{equation*}
\left\|F'w\right\|_{L_p(\mathbb{R})}= E_{p,n-1}\left(w,f'\right),
\end{equation*}
\begin{equation*}
 \|(F-S_{2n})w^{\sharp}\|_{L_p(\mathbb{R})}\le C  \frac{a_n}{n}E_{p,n}\left(w_{1/4},f'\right)
\end{equation*}
and
\begin{equation}\label{4.7}
 \left\|S_{2n}'w \right\|_{L_p(\mathbb{R})}\le C E_{p,n-1}\left(w_{1/4},f'\right).
\end{equation}
Then we see from Theorem \ref{Theorem 3.9} and  (\ref{4.7}),
\begin{equation}\label{4.8}
 \left\|S_{2n}^{(k)}w_{-(k-1)/2}\right\|_{L_p(\mathbb{R})}
\le C \left(\frac{n}{a_n}\right)^{k-1}E_{p,n-1}\left(w_{1/4},f'\right)
\end{equation}
and using $w^{\sharp} \le w$ and Theorem \ref{Theorem 3.8}
\begin{equation}\label{4.9}
\left\|\left(R_{n}^{(k)}-S_{2n}^{(k)}\right) w^{\sharp}_{-k/2}\right\|_{L_p(\mathbb{R})}
\le C \left(\frac{n}{a_n}\right)^{k-1}E_{p,n-1}\left(w_{1/4},f'\right),
\end{equation}
where $R_{n}\in \mathcal{P}_{n}$ denotes the polynomial of best approximation of  $F$
with $w$ on $L_p(\mathbb{R})$ space(by the similar calculation as (\ref{3.20}) and (\ref{3.21})).
Then, we see $w^{\sharp}_{-k/2}(x) \le w_{-(k-1)/2}(x)$.
By (\ref{4.8}) and (\ref{4.9}) and Theorem \ref{Theorem 3.8}, we have
\begin{equation}\label{4.10}
\left\|R_{n}^{(k)}w^{\sharp}_{-k/2}\right\|_{L_p(\mathbb{R})}
\le C \left(\frac{n}{a_n}\right)^{k-1}E_{p,n-1}\left(w_{1/4},f'\right)
\le C E_{p,n-k}\left(w_{1/4},f^{(k)}\right).
\end{equation}
By the same reason to (\ref{3.24}), we know that
$P_{p,n;f,w}:=Q_n+R_n$ is the polynomial of best approximation of $f$ with $w$ on $L_p(\mathbb{R})$ space.
Therefore,
using $P_{p,n;f,w}$, (\ref{4.10}) and the method of mathematical induction,
we have for $1\le k\le r+1$,
\begin{equation*}
\left\|\left(f^{(k)}-P_{p,n;f,w}^{(k)}\right)w^{\sharp}_{-k/2}\right\|_{L_p(\mathbb{R})}
\le C E_{p,n-k}\left(w_{1/4},f^{(k)}\right).
\end{equation*}
\end{proof}
\begin{proof}[Proof of Corollary \ref{Corollary 4.2}]
It follows from Theorem \ref{Theorem 4.1}.
\end{proof}
\begin{proof}[Proof of Corollary \ref{Corollary 4.3}]
If we apply Theorem \ref{Theorem 4.1} with $w_{k/2}$ and $w^{\flat}_{k/2}$,
then we can obtain the results.
\end{proof}

\setcounter{equation}{0}
\section{Monotone Approximation}

Let $r>0$ be an integer. Let $k$ and $\ell$ be integers with $0\le k\le \ell\le r$.
In this section, we consider a real function $f$ on $\mathbb{R}$
such that $f^{(r)}(x)$ is continuous in $\mathbb{R}$
and we let $a_j(x)$,   $j=k,k+1,...,\ell$ be bounded on $\mathbb{R}$.

Now, we define the linear differential operator (cf. \cite{[1]})
\begin{equation}\label{5.1}
L:=L_{k,\ell}:=\sum_{j=k}^\ell a_j(x)[d^j/dx^j].
\end{equation}
G. A. Anastassiou and O. Shisha \cite{[1]} consider the operator (\ref{5.1}) with $a_j(x)$ under some condition
on $[-1,1]$. They showed that if $L(f) \ge 0$ for $f\in C^{(r)}[-1,1]$,
there exist $Q_n \in \mathcal{P}_n$ such that $L(Q_n)\ge 0$ and for some constant $C>0$,
\begin{equation*}
\left\|f-Q_n\right\|_{L_{\infty}([-1,1])}
\le C n^{\ell -r}\omega\left(f^{(p)};\frac1n\right),
\end{equation*}
where $\omega\left(f^{(p)};t\right)$ is the modulus of continuity.
In this section, we will obtain a similar result  with exponential-type weighted $L_{\infty}$-norm as the above result.
Our main theorem is as follows.
\begin{theorem}\label{Theorem 5.1}
Let $k$ and $\ell$ be integers with $0\le k\le \ell\le r$.
Let $w=\exp(-Q)\in \mathcal{F}_\lambda(C^{r+3}+)$,
and let $T(x)$ be continuous on $\mathbb{R}$. Suppose that
$w(x)f^{(r)}(x)\to 0$ as $|x|\to \infty$.
Let $P_{n;f,w_{-1/4}}\in\mathcal{P}_n$ be the best approximation for $f$ with the weight $w_{-1/4}$ on $\mathbb{R}$.
Suppose that for a certain $\delta>0$,
\begin{equation*}
 L(f;x)\ge \delta, \quad  x\in\mathbb{R}.
\end{equation*}
Then, for every integer $n\ge 1$ and $j=0,1,...,\ell$,
\begin{equation}\label{5.2}
\left\|\left(f^{(j)}-P_{n;f,w_{-1/4}}^{(j)}\right)wT^{-(2j+1)/4}(x)\right\|_{L_\infty(\mathbb{R})}
\le C_j\left(\frac{a_n}{n}\right)^{r-j}E_{n-r}\left(w,f^{(r)}\right),
\end{equation}
where $C_j>0$,   $0\le j\le \ell$, are independent of $n$ or $f$,
and for any fixed number $M>0$ there exists a constant $N(M,\ell,\delta)>0$ such that
\begin{equation}\label{5.3}
 L(P_{n;f,w_{-1/4}};x)\ge \frac{\delta}{2}, \quad |x|\le M, \quad   n\ge N(M,\ell,\delta).
\end{equation}
\end{theorem}
\begin{proof}
From Corollary \ref{Corollary 2.4}, we have (\ref{5.2}). Hence, we also have
\begin{eqnarray*}
&&\left|\left(L(f;x)-L\left(P_{n;f,w_{-1/4}};x\right)\right)w(x)T^{-(2\ell+1)/4}(x)\right|\\
&=&\left|\sum_{j=k}^\ell a_j(x)\left\{f^{(j)}(x)-P_{n;f,w_{-1/4}}^{(j)}(x)\right\}w(x)T^{-(2\ell+1)/4}(x)\right|\\
&\le& E_{n-r}\left(w,f^{(r)}\right)\sum_{j=k}^\ell |a_j(x)|C_j\left(\frac{a_n}{n}\right)^{r-j}\\
&\le& C_{k,\ell}\left(\frac{a_n}{n}\right)^{r-\ell}E_{n-r}\left(w,f^{(r)}\right),
\end{eqnarray*}
where we set $C_{k,\ell}:=\sum_{j=k}^\ell \|a_j\|_{L_\infty(\mathbb{R})}C_j$. Then we have for $|x|\le M$
\begin{eqnarray*}
&&\left|L(f;x)-L\left(P_{n;f,w_{-1/4}};x\right)\right|\\
&\le& C_{k,\ell}\left\|w^{-1}(x)T^{(2\ell+1)/4}(x)\right\|_{L_\infty(|x|\le M)}
\left(\frac{a_n}{n}\right)^{r-\ell}E_{n-r}\left(w,f^{(r)}\right).
\end{eqnarray*}
Here, for $\delta>0$ there exists $N(M,\ell,\delta)>0$ such that for $n\ge N(M,\ell,\delta)$
\begin{eqnarray*}
C_{k,\ell}\left\|w^{-1}(x)T^{(2\ell+1)/4}(x)\right\|_{L_\infty(|x|\le M)}
\left(\frac{a_n}{n}\right)^{r-\ell}E_{n-r}\left(w,f^{(r)}\right)
\le \frac{\delta}{2}.
\end{eqnarray*}
This follows from $w(x)f^{(r)}(x)\to 0$ as $|x|\to \infty$. Therefore we see
\begin{equation*}
 \frac{\delta}{2}\le L(f;x)-\frac{\delta}{2}\le L(P_{n;f,w_{-1/4}};x).
\end{equation*}
Consequently, we have (\ref{5.3}).
\end{proof}


\end{document}